\documentclass[12pt]{amsart}
\usepackage{etex}

\usepackage{amsthm,amssymb,enumitem,mathtools}

\usepackage{tikz}

\usepackage{scalerel}

\usetikzlibrary{arrows,shapes,automata,backgrounds,decorations,petri,positioning}
\usetikzlibrary[calc,intersections,through,backgrounds,arrows,decorations.pathmorphing]

\tikzstyle{edge} = [fill,opacity=.5,fill opacity=.5,line cap=round, line join=round, line width=50pt]

\pgfdeclarelayer{background}
\pgfsetlayers{background,main}

\usepackage[margin=1in,letterpaper,portrait]{geometry}

\usepackage{amsthm}

\usepackage[latin1]{inputenc}
\usepackage{subfigure}
\usepackage{color}
\usepackage{amsmath}
\usepackage{amsthm}
\usepackage{amstext}
\usepackage{amssymb}
\usepackage{amsfonts}
\usepackage{graphicx}
\usepackage{young}
\usepackage{multicol}
\usepackage{mathrsfs}
\usepackage[all]{xy}
\usepackage{tikz}
\usepackage{mathtools}
\usepackage{tabularx}
\usepackage{array}
\usepackage{commath}
\usepackage{ytableau}

\usepackage{scrextend}

\usepackage{hyperref}

\theoremstyle{plain}
\theoremstyle{definition}
\newtheorem*{unnum}{Theorem}
\newtheorem{theorem}{Theorem}[section]

\newtheorem{remark}[theorem]{Remark}
\newtheorem{lemma}[theorem]{Lemma}

\newtheorem{definition}[theorem]{Definition}

\newtheorem{example}[theorem]{Example}

\newtheorem{corollary}[theorem]{Corollary}

\DeclareMathAlphabet{\mathpzc}{OT1}{pzc}{m}{it}

\newcommand{\symm}{\mathfrak{S}}
\newcommand{\s}{\sigma}

\newcommand{\supp}{\textnormal{\textsf{supp}}}
\newcommand{\pos}{\textnormal{\textsf{PosPair}}}
\newcommand{\val}{\textnormal{\textsf{ValPair}}}
\newcommand{\expat}{\textnormal{\textsf{expat}}}
\newcommand{\ub}[1]{\underbracket[.5pt][1pt] {#1}}

\begin{document}

\title{The range of repetition in reduced decompositions}%

\date{}

\author{Bridget Eileen Tenner}
\address{Department of Mathematical Sciences, DePaul University, Chicago, IL, USA}
\email{bridget@math.depaul.edu}
\thanks{Research partially supported by Simons Foundation Collaboration Grant for Mathematicians 277603 and by a University Research Council Competitive Research Leave from DePaul University. This material is based partly upon work supported by the Swedish Research Council under grant no. 2016-06596 while the author was in residence at Institut Mittag-Leffler in Djursholm, Sweden during the winter of 2020.}



\begin{abstract}%
Given a permutation $w$, we look at the range of how often a simple reflection $\s_k$ appears in reduced decompositions of $w$. We compute the minimum and give a sharp upper bound on the maximum. That bound is in terms of $321$- and $3412$-patterns in $w$, specifically as they relate in value and position to $k$. We also characterize when that minimum and maximum are equal, refining a previous result that braid moves are equivalent to $321$-patterns.
\end{abstract}

\maketitle

There is an intimate relationship between occurrences of the patterns $321$ and $3412$, and repeated letters in reduced decompositions. These two patterns (entry P0006 of \cite{dppa}), more than any others, seem to determine the structure of the Bruhat order of the symmetric group, and have arisen in numerous papers \cite{bjs, petersen tenner, ragnarsson tenner 1, ragnarsson tenner 2, tenner patt-bru, tenner repetition}. In \cite{tenner patt-bru}, we showed that having no repeated letters was equivalent to avoiding $321$ and $3412$. That instigated the study of boolean elements, as in \cite{claesson kitaev ragnarsson tenner, hultman vorwerk, ragnarsson tenner 1, ragnarsson tenner 2, ragnarsson tenner 3}. In \cite{tenner repetition}, we substantially generalized that repetition result, showing the following.

\begin{unnum}[{\cite[Theorem 3.2]{tenner repetition}}]
For any permutation $w$, the number of repeated letters in any reduced decomposition of $w$ is less than or equal to the total number of $321$ and $3412$ patterns in $w$. Moreover, the quantities are equal if and only if $w$ avoids the ten patterns
$$\{4321,\ 34512,\ 45123,\ 35412,\ 43512,\ 45132,\ 45213,\ 53412,\ 45312,\ 45231\}.$$
\end{unnum}

In the present work, we refine those efforts still further, to describe the range of repetition that each letter, individually, may have in reduced decompositions of a permutation. Note that this is not just for the long element. Note also that our results do not just refer to the existence of $321$ and $3412$ patterns in $w$, but rather to the number of occurrences of these patterns and to their locations in $w$. Both of these features of the patterns have been notoriously hard to work with and have been rarely utilized in the existing literature. Given the focus on individual letters in reduced decompositions, our results might call to mind Macdonald's formula \cite{macdonald}, and its recent bijective proof due to Billey, Holroyd, and Young \cite{billey holroyd young}, which looks at the individual factors appearing, with multiplicity, in each reduced decomposition of an element.

Throughout this work, we assume that the reader is familiar with basic concepts like permutation patterns, simple reflections, and Coxeter relations. We refer the reader to texts like \cite{bjorner brenti, kitaev} for more information.

Let $\min_k(w)$ and $\max_k(w)$ be the minimum and maximum number of times $\s_k$ can appear in a reduced decomposition of $w$, respectively. 
Even in $\symm_3$ it is obvious that $\min_k(w)$ and $\max_k(w)$ need not be equal. However, an understanding of these values and their relationship to the permutation $w$ has previously been elusive. Our main results show that
\begin{itemize}
\item $\min_k(w)$ is equal to measure of ``expatriation'' (Theorem~\ref{thm:minimal}), 
\item having $\max_k(w) > 1$ is equivalent to having a $321$- or $3412$-pattern that ``straddles'' $k$ in a particular way (Theorem~\ref{thm:max>1}),
\item $\max_k(w)$ is bounded by the number of $321$ and $3412$ patterns ``straddling'' $k$ (Theorem~\ref{thm:maximal}), and
\item a characterization of when $\min_k(w) = \max_k(w)$ (Theorem~\ref{thm:fixed}).
\end{itemize}

This paper is organized as follows. In Section~\ref{sec:terminology}, we introduce key notation and terminology relevant to our work, with the understanding that the reader is referred to other sources for basic definitions. The question of minimally many appearances is studied in Section~\ref{sec:minimal}, while maximally many appearances are covered in Section~\ref{sec:maximal}. That latter section will rely on the concept of ``straddling'' patterns and ``pairs,'' and these are covered in Section~\ref{sec:straddling}, along with fundamental results justifying their inclusion in this work. In Section~\ref{sec:fixed}, we characterize when all reduced decompositions of a permutation contain the same number of $\s_k$ factors. That result is particularly intriguing because requiring a fixed number of $\s_k$ factors means avoiding any braid factors using $\s_k$, and braid factors in general are equivalent to $321$-patterns \cite{bjs, tenner rdpp, tenner rwm}. We conclude with a sampling of open questions in Section~\ref{sec:open}.

\section{Notation and terminology}\label{sec:terminology}

We write $\symm_n$ for the permutations of $[1,n]$, and $\s_i$ for the simple reflection exchanging $i$ and $i+1$. We consider permutations as maps, and so $w\s_i$ transposes the values in positions $i$ and $i+1$ in $w$, whereas $\s_iw$ transposes the positions of the values of $i$ and $i+1$ in $w$. For example,
$$\s_1\s_2  = 231 \in\symm_3.$$
The set of reduced decompositions of $w$ is denoted $R(w)$, and the \emph{support} of $w$, denoted $\supp(w)$, is the collection of distinct letters appearing in any element of $R(w)$. Although we will not use it here, we note that $[1,n-1]\setminus\supp(w)$ is the \emph{connectivity set} introduced in \cite{rps connectivity} and generalized in \cite{marietti}. 

For the remainder of the paper, we assume that $w \in \symm_n$, and $k \in [1,n-1]$ is fixed.

We are interested in understanding how often a particular $\s_k$ can appear in elements of $R(w)$. This is different from the results of \cite{tenner repetition}, which characterized overall repetition, not repetition of an individual letter. To discuss this, we make the following definitions, as suggested above.

\begin{definition}
Let
$$\min\nolimits_k(w)$$
be the minimum number of times that $\s_k$ appears in any element of $R(w)$, and let
$$\max\nolimits_k(w)$$
be the maximum number of times that $\s_k$ appears in any element of $R(w)$.
\end{definition}

The definitions suggest that the number of $\s_k$ factors in a reduced decomposition of $w$ need not be fixed, and indeed that is the case. One need only look so far as the long element in $\symm_3$ to see an example of this, and we give a slightly more interesting example below.

\begin{example}\label{ex:4312}
Let $w = 4312$, and so
$$R(w) = \{\s_2\s_3\s_2\s_1\s_2,\ \s_3\s_2\s_3\s_1\s_2,\ \s_3\s_2\s_1\s_3\s_2,\ \s_2\s_3\s_1\s_2\s_1,\ \s_2\s_1\s_3\s_2\s_1\}.$$
The ranges of repetition of the letters $\s_1$, $\s_2$, and $\s_3$ are recorded in Table~\ref{table:4312 reduced decompositions}.
\end{example}
\begin{table}[htbp]
$\begin{array}{c||c|c}
\raisebox{0in}[.2in][.1in]{\ $k$\ } & \min_k(w) & \max_k(w)\\
\hline
\hline
\raisebox{0in}[.2in][.1in]{$1$} & 1 & 2\\
\hline
\raisebox{0in}[.2in][.1in]{$2$} & 2 & 3\\
\hline
\raisebox{0in}[.2in][.1in]{$3$} & 1 & 2
\end{array}$
\caption{Ranges of repetition for reduced decompositions of $4312$.}\label{table:4312 reduced decompositions}
\end{table}

As stated above, occurrences of the patterns $321$ and $3412$ are closely linked with repeated letters in the permutation's reduced decomposition. This was first suggested in \cite{tenner patt-bru}, expanded further by Daly in \cite{daly}, and broadly proved in \cite{tenner repetition}. In the latter work, we mapped repeated letters in the reduced decomposition (that is, appearances of simple reflections that were not the first appearances of that letter) to occurrences of $321$ and $3412$ in the permutation. The distinction between a \emph{repeat} of a letter and an \emph{occurrence} of that letter warrants a pause. For example, in the reduced decomposition $\s_2\s_3\s_2\s_1\s_2 \in R(4312)$, the letter $\s_2$ appears three times, but it repeats twice. In other words, in any product,
\begin{equation}\label{eqn:occurrences = repeats + 1}
\# \text{ occurrences of } \s_k = \# \text{ repeats of } \s_k + 1,
\end{equation}
and so understanding the number of occurrences of a letter is equivalent to understanding the amount of repetition of that letter.

Our goal in this paper is to understand how often a given letter can appear, among all elements of $R(w)$. Thus we are interested in the number of \emph{occurrences} of that letter. Because of the relationship described in Equation~\eqref{eqn:occurrences = repeats + 1}, this will, in a sense, refine the main result of \cite{tenner repetition}.  That refinement will require more precise language for talking about the permutation patterns $321$ and $3412$.

\begin{definition}\label{defn:straddling}\
\begin{itemize}
\item If $w$ has a $321$-pattern in positions $i_1 < i_2 < i_3$ with $i_1 \le k < i_3$, then this occurrence \emph{straddles $k$ in position} and $(i_1,i_3)$ is a \emph{position pair} at $k$. 
\item If $w$ has a $321$-pattern with values $j_1 < j_2 < j_3$ such that $j_1 \le k < j_3$, then this occurrence \emph{straddles $k$ in value} and $(j_3,j_1)$ is a \emph{value pair} at $k$. 
\item If $w$ has a $3412$-pattern in positions $i_1 < i_2 \le k < i_3 < i_4$, then this occurrence \emph{straddles $k$ in position} and $(i_2,i_3)$ is a \emph{position pair} at $k$.
\item If $w$ has a $3412$-pattern with values $j_1 < j_2 \le k < j_3 < j_4$, then this occurrence \emph{straddles $k$ in value} and $(j_4,j_1)$ is a \emph{value pair} at $k$.
\end{itemize}
Each position/value pair is said to \emph{mark} its corresponding pattern(s).
\end{definition}

The idea of position (respectively, value) straddling is that roughly half of the occurrence is weakly to the left of the $k$th position (resp., is less than or equal to $k$), and roughly half of the occurrence is to the right of the $k$th position (resp., is greater than $k$). The idea of a straddling pair is to identify the largest and smallest letters in a straddling pattern.

\begin{remark}\label{rem:straddling pair marks extremes}
In a position pair $(x,y)$, the largest value in the pattern (whether $321$ or $3412$) is in position $x$ and the smallest value is in position $y$. In a value pair $(x,y)$, the largest value in the pattern is $x$ and the smallest is $y$.
\end{remark} 

It is possible for a pair to mark more than one pattern, and a single pair could mark both a $321$-pattern and a $3412$-pattern.

\begin{example}
Consider $w = 5273416$. The $321$-patterns in $w$ are $521$, $531$, $541$, $731$, and $741$, and the lone $3412$-pattern in $w$ is $5734$. These patterns straddle various positions and values, as catalogued in Table~\ref{table:straddling example}.
\end{example}
\begin{table}[htbp]
$\begin{array}{c||l|l||l|l}
& \text{Patterns} & & \text{Patterns} &\\
& \text{straddling} & \text{Position} & \text{straddling} & \text{Value}\\
\ \ k \ \ & \text{position }k & \text{pairs at }k & \text{value }k & \text{pairs at }k\\
\hline
\hline
\raisebox{0in}[.2in][.1in]{$1$} & 521,\ 531,\ 541 & (1,6) 
& 521,\ 531,\ 541, & (5,1),\ (7,1) \\
& & & 731,\ 741 &\\
\hline
\raisebox{0in}[.2in][.1in]{$2$} & 521,\ 531,\ 541 & (1,6) 
& 521,\ 531,\ 541, & (5,1),\ (7,1) \\
& & & 731,\ 741 &\\
\hline
\raisebox{0in}[.2in][.1in]{$3$} & 521,\ 531,\ 541 & (1,6),\ (3,6),\ (3,4)
& 521,\ 531,\ 541, & (5,1),\ (7,1)\\
& 731,\ 741,\ 5734 & & 731,\ 741\\
\hline
\raisebox{0in}[.2in][.1in]{$4$} & 521,\ 531,\ 541, & (1,6),\ (3,6)
& 521,\ 531,\ 541, & (5,1),\ (7,1),\ (7,3)\\
& 731,\ 741 & & 731,\ 741,\ 5734\\
\hline
\raisebox{0in}[.2in][.1in]{$5$} & 521,\ 531,\ 541, &(1,6),\ (3,6)
& 731,\ 741 & (7,1)\\
& 731,\ 741 & & &\\
\hline
\raisebox{0in}[.2in][.1in]{$6$} & - & -
& 731,\ 741 & (7,1)\\
\end{array}$
\caption{The straddling patterns and straddling pairs that appear in $5273416$.}\label{table:straddling example}
\end{table}

In \cite{tenner repetition}, we counted $321$ and $3412$ patterns in $w$. Here, not surprisingly, we want to be more specific in terms of straddling.

\begin{definition}
Let
$$\pos_k(w)$$
count the straddling position pairs at $k$ in $w$, and
$$\val_k(w)$$
count the straddling value pairs at $k$ in $w$.
\end{definition}

\begin{example}
Continuing the example $w = 5273416$, we can compute $\pos_k(w)$ and $\val_k(w)$ as shown in Table~\ref{ex:pair counts}.
\end{example}
\begin{table}[htbp]
$\begin{array}{c|c|c}
\raisebox{0in}[.2in][.1in]{\ $k$ \ } & \pos_k(w) & \val_k(w)\\
\hline
\hline
\raisebox{0in}[.2in][.1in]{$1$} & 1 & 2\\
\hline
\raisebox{0in}[.2in][.1in]{$2$} & 1 & 2\\
\hline
\raisebox{0in}[.2in][.1in]{$3$} & 3 & 2\\
\hline
\raisebox{0in}[.2in][.1in]{$4$} & 2 & 3\\
\hline
\raisebox{0in}[.2in][.1in]{$5$} & 2 & 1\\
\hline
\raisebox{0in}[.2in][.1in]{$6$} & 0 & 1
\end{array}$
\caption{Counting position pairs and value pairs in $5273416$.}\label{ex:pair counts}
\end{table}

\section{Minimal repetition}\label{sec:minimal}

In \cite[Lemma 2.8]{tenner repetition}, we showed, among other things, that $\min_k(w) \ge 1$ if and only if $\{w(1),\ldots,w(k)\} \neq \{1,\ldots, k\}$. This idea of \emph{expatriation}---that one of the smallest $k$ values has been moved out of the first $k$ positions or, equivalently, that one of the largest $n-k$ values has been moved into those first $k$ positions for the first time---can actually be used to understand $\min_k(w)$ entirely, not just to bound it.

\begin{definition}
The \emph{expatriation} measure of $w$ at $k$ is
\begin{align*}
\expat_k(w) &:=  \Big| \big\{w(1),\ldots,w(k)\big\} \cap \big\{k+1,\ldots, n\big\} \Big|\\
&\phantom{:}= \Big| \big\{w(k+1),\ldots,w(n)\big\} \cap \big\{1,\ldots, k\big\} \Big|.
\end{align*}
\end{definition}

The following result shows an interaction between the positions and values in a permutation, nicely echoing the symmetry of these features that one often exploits in permutation analysis. This symmetry will reappear in the bound for maximal repetition, appearing in Theorem~\ref{thm:maximal}.

\begin{theorem}\label{thm:minimal}
Fix $w \in \symm_n$ and $k \in [1,n-1]$. Then
$$\min\nolimits_k(w) = \expat_k(w).$$
\end{theorem}

\begin{proof}
The set $X_k := \{w(1),\ldots,w(k)\} \cap \{k+1,\ldots, n\}$ describes the ``large'' values (greater than $k$) that have crossed into ``small'' positions (less than or equal to $k$). In any product of simple reflections, read from left to right, this kind of expatriation requires the reflection $\s_k$. Moreover, an individual $\s_k$ can only increase the expatriation by $1$, so
$$\min\nolimits_k(w) \ge |X_k| = \expat_k(w).$$

To show inequality in the other direction, we demonstrate a reduced decomposition of $w$ having exactly $|X_k|$ copies of $\s_k$. Set $Y_k := \{w(k+1),\ldots,w(n)\} \cap \{1,\ldots, k\}$. Let $u \in \symm_n$ be the permutation that puts the letters of $Y_k$, in increasing order, immediately to the right of the letters $\{1,\ldots,k\} \setminus \{w(k+1),\ldots,w(n)\}$. Let $v \in \symm_n$ be the permutation that puts the letters of $X_k$, in increasing order, immediately to the left of the letters $\{k+1,\ldots,n\}\setminus\{w(1),\ldots,w(k)\}$. Note that $\s_k \not\in \supp(u) \cup \supp(v)$.

For $i \in [0,|X_k|-1]$, define the permutation
$$t_i := \s_{k+i}\s_{k+i-1}\cdots\s_{k+i-(|X_k|-1)},$$
which is, in fact, given as a reduced decomposition. The permutation
$$uvt_0t_1\cdots t_{|X_k|-1}$$
differs from $w$ only the ordering of its first $k$ elements and, separately, of its last $n-k$ elements. Let $d$ be the permutation acting on positions $\{1,\ldots, k\}$ and, separately, $\{k+1,\ldots,n\}$ so that
$$uvt_0t_1\cdots t_{|X_k|-1}d = w.$$
By \cite[Lemma 2.8]{tenner repetition}, we have $\s_k \not\in \supp(u) \cup \supp(v) \cup \supp(d)$. On the other hand, $\s_k$ appears exactly once in $t_i$ for each $i \in [0,|X_k| - 1]$. Thus $\min_k(w) \le |X_k| = \expat_k(w)$, completing the proof.
\end{proof}

By construction, the proof of Theorem~\ref{thm:minimal} gives a parabolic decomposition of $w$, with $d \in W_J$ and $uvt_0t_1\cdots t_{|X_k|-1} \in W^J$, for the set $J$ of all generators except $\s_k$. Moreover, the size of the Durfee square of the partition defined by the Lehmer code of that minimal coset representative is, in fact, the amount of expatriation at $k$ in $w$.

We demonstrate the construction in the proof of Theorem~\ref{thm:minimal} with an example.

\begin{example}
Consider $w = 5273416$ and $k = 4$. Then $\expat_4(w) = 2$, with $X_4 = \{5,7\}$ and $Y_4 = \{1,4\}$. As described in the proof of Theorem~\ref{thm:minimal}, we find permutations
$$u = 2314567 \text{ \ and \ } v = 1234576,$$
as well as
$$t_0 = \s_4\s_3 \text{ \ and \ } t_1 = \s_5 \s_4.$$
This produces the permutation
$$uvt_0t_1 = 2357146.$$
Thus $d = 3142657$, from which we get
$$(uvt_0t_1)d = w.$$
We have $\s_4 \not\in \supp(u)\cup\supp(v)\cup\supp(d)$, while $\s_4$ appears once in each of $t_0$ and $t_1$. Therefore $\min_4(w) \le 2$. Because $|X_4| = 2 \le \min_4(w)$, this gives $\min_4(w) = 2$.
\end{example}

\section{Straddling patterns}\label{sec:straddling}

We now establish the relevance of straddling pairs by showing how the appearance of repeated factors in reduced decompositions can start to affect the straddling pairs in the permutation.

\begin{lemma}\label{lem:max=1}
If $\max_k(w) = 1$, then $w$ has no $3412$-pattern straddling $k$ in position or in value, and no $321$-pattern straddling $k$ in both position and value.
\end{lemma}

\begin{proof}
If $\max_k(w) = 1$, then $\min_k(w) = 1$ as well. Thus, by Theorem~\ref{thm:minimal}, there is exactly one expatriated value in the first $k$ positions of $w$, and one in the last $n-k$ positions. Thus there can be no $3412$-pattern straddling $k$ in position or value. Let these expatriated values be $w^+$ and $w^-$, respectively. Because $w^+$ and $w^-$ are the lone expatriated values relative to $k$, there is a $321$-pattern in $w$ straddling $k$ in both position and value if and only if $w^+$ and $w^-$ are both part of that pattern (and, in fact, the largest and smallest values in it).

Suppose that there is such a pattern. Let $z$ be the central letter in the pattern, and suppose, without loss of generality, that $z \in \{w(1),\ldots,w(k)\}$. If necessary, redefine $z$ so that it is the rightmost letter in $\{w(1),\ldots,w(k)\}$ that is greater than $w^-$.

Using techniques as in \cite{tenner rdpp, tenner rwm}, we will multiply $w$ on the right by simple reflections, always shortening the length, to obtain $v$ in which there is a $321$-pattern now in positions $k-1 < k< k+1$ (in the case $w^{-1}(z) > k$, we would find the pattern in positions $k < k+1< k+2$). More precisely, write $w$ in one-line notation as
$$w = \ub{\cdots \phantom{j}\! w^+ \ A \ z \ B}_k \ \ub{C \ w^- \phantom{j}\! \cdots}_{n-k}.$$
All letters of $A$ are less than $w^+$, all letters of $B$ are less than $w^-$, and all letters of $C$ are greater than $w^-$. Thus $w = vu$ where $\ell(w) = \ell(v) + \ell(u)$ and
$$v = \ub{\cdots \phantom{j}\! A \ B \ w^+ \ z}_k \ \ub{w^- \ C \phantom{j}\! \cdots}_{n-k}.$$
This $v$ has a $321$-pattern in the consecutive positions $k-1, k, k+1$, and so it has reduced decompositions with the factor $\s_k\s_{k-1}\s_k$, and others that differ only by changing that factor to $\s_{k-1}\s_k\s_{k-1}$. Therefore $w$ has such reduced decompositions as well, and hence $\max_k(w) \ge 2$, a contradiction.

Therefore, no $321$-pattern in $w$ uses both $w^+$ and $w^-$, and so $w$ has no $321$-pattern straddling $k$ in both position and value.
\end{proof}

\begin{example}
Let $w = 621354$. The reduced decompositions of $w$, such as $\s_4\s_5\s_4\s_3\s_2\s_1\s_2$, each contain exactly one $\s_3$ factor. There are no $3412$-patterns in $w$. The only $321$-patterns in $w$ are $621$ and $654$. The former straddles $3$ in value, but not position. The latter straddles $3$ in position, but not value.
\end{example}

In fact, Lemma~\ref{lem:max=1} hints at the difference between having one $\s_k$ factor in a reduced decomposition, and having more than one factor.

\begin{theorem}\label{thm:max>1}
Fix $w \in \symm_n$ and $k \in [1,n-1]$. Then $\max_k(w) > 1$ if and only if $w$ has a pattern that straddles $k$ in both position and value.
\end{theorem}

\begin{proof}
If $\max_k(w) = 0$ then \cite[Lemma 2.8]{tenner repetition} says that $w$ can have no such pattern. The case $\max_k(w) = 1$ was handled in Lemma~\ref{lem:max=1}.

Now suppose that $\max_k(w) > 1$. We want to find such a pattern straddling $k$ in $w$. We induct on the length of $w$, reading a reduced decomposition from left to right and showing that whenever it is multiplied by some $\s_h$, we can identify such a pattern in the resulting product.

Consider, first, a permutation $v$ and a longer permutation $v' := v\s_k$, such that $\s_k \in \supp(v)$. By \cite[Lemma 2.8]{tenner repetition}, we can find $i$ and $j$ such that $i \le k < j$ and $v(i) > k \ge v(j)$. If at least one of the values $x \in \{v(k),v(k+1)\}$ satisfies $v(i) > x > v(j)$, then this is a $321$-pattern in $v'$ that straddles $k$ in both position and value. If that is not the case, then, because $\ell(v') > \ell(v)$, we have one of the following situations:
\begin{itemize}
\item $v(i) \le v(k) < v(k+1)$, in which case $v(k+1) > v(k) > v(j)$ is a $321$-pattern in $v'$ that straddles $k$ in both position and value; 
\item $v(j) \ge v(k+1) > v(k)$, in which case $v(i) > v(k+1) > v(k)$ is a $321$-pattern in $v'$ that straddles $k$ in both position and value; or
\item $v(i) < v(k+1)$ and $v(j) > v(k)$, in which case we have $v(k+1) > v(i) > v(j) > v(k)$, appearing in the order $v(i) v(k+1) v(k) v(j)$ in $v'$, which is a $3412$-pattern that straddles $k$ in both position and value.
\end{itemize}

Now suppose that we have read the reduced decomposition up to a certain point, producing a permutation $v$, with $\max_k(v) > 1$, and suppose that the next reflection is $\s_h$ for $h \neq k$. Suppose $\{x,y\} \cap \{h,h+1\} = \emptyset$. Then, up to symmetry, we can trace all position pairs from $v$ to $v' = v\s_h$ as follows.
\begin{center}
\begin{tabular}{lll}
\raisebox{0in}[.2in][.1in]{}Position pair in $v$ & & Position pair in $v'$\\
\hline
\hline
\raisebox{0in}[.2in][.1in]{}$(x,y)$ marking $321$ & $\mapsto$ & $(x,y)$ marking $321$\\
\hline
\raisebox{0in}[.2in][.1in]{}$(x,y)$ marking $3412$ & $\mapsto$ & $(x,y)$ marking $3412$\\
\hline
\raisebox{0in}[.2in][.1in]{}$(x,h)$ marking $321$ & $\mapsto$ & $(x,h+1)$ marking $321$\\
\hline
\raisebox{0in}[.2in][.1in]{}$(x,h+1)$ marking $321$ & $\mapsto$ & $(x,h)$ marking $321$\\
\hline
\raisebox{0in}[.2in][0in]{}$(x,h)$ marking $3412$ with & $\mapsto$ & $(x,h+1)$ marking $3412$\\
\raisebox{0in}[0in][.1in]{}``$2$'' not in position $h+1$ && \\
\hline
\raisebox{0in}[.2in][0in]{}$(x,h)$ marking $3412$ with & $\mapsto$ & $(x,h+1)$ marking $321$\\
\raisebox{0in}[0in][.1in]{}``$2$'' in position $h+1$ && \\
\hline
\raisebox{0in}[.2in][.1in]{}$(x,h+1)$ marking $3412$ & $\mapsto$ & $(x,h)$ marking $3412$
\end{tabular}
\end{center}
Thus such a pair always persists in $v' = v\s_h$.
\end{proof}

\section{Maximal repetition}\label{sec:maximal}

Unfortunately, the maximal amount of repetition of $\s_k$ in any element of $R(w)$ cannot be characterized as nicely as Theorem~\ref{thm:minimal} did for the minimal amount. In particular, we achieve a sharp upper bound for $\max_k(w)$ instead of an exact equality. Perhaps this should not be surprising, given that the main result of \cite{tenner repetition} was a sharp bound and not a strict equality.

We can now bound $\max_k(w)$. This is a tighter bound than what one might have guessed from \cite{tenner repetition}; namely, the bound is in terms of straddling \emph{pairs} at $k$, not straddling \emph{patterns}.

\begin{theorem}\label{thm:maximal}
Fix $w \in \symm_n$ and $k \in [1,n-1]$. Then
\begin{equation}\label{eqn:max}
\max\nolimits_k(w) \le \min\big\{\val_k(w),\pos_k(w)\big\} + 1.
\end{equation}
\end{theorem}

\begin{proof}
We will prove that the number of $\s_k$ factors in any reduced decomposition of $w$ is bounded by $\pos_k(w) + 1$. The argument for $\val_k(w) + 1$ is similar. Thus, we will have shown that the number of $\s_k$ factors in any reduced decomposition of $w$ is bounded by $\min\{\pos_k(w),\val_k(w)\} + 1$, and the result follows.

We prove this by induction on $\max_k(w)$. For the remainder of the proof, we will take ``position pair'' to mean ``position pair at $k$.''

Suppose, first, that $\max_k(w) = 1$. Since $\min\{\pos_k(w),\val_k(w)\} \ge 0$, the result is trivial. 

Now consider a permutation $w$ for which $\max_k(w) > 1$, and consider a reduced decomposition of $w$. As before, we will read the decomposition from left to right, showing that each $\s_k$ after the first one will produce at least one new position pair, and that adjusting the permutation by multiplying other simple reflections on the right will not reduce the number of these pairs.

Suppose, first, that we have read the reduced decomposition up to a certain point, producing a permutation $v$, and $v' = v\s_h$ for $h \neq k$. Then the table presented in the proof Theorem~\ref{thm:max>1} is again relevant, and the mapping from position pairs in $v$ to position pairs in $v'$, given by $(x,y) \mapsto (\s_h(x),\s_h(y))$, is injective.

Now suppose that we have read the reduced decomposition up to a certain point, producing a permutation $v$, and the next reflection is $\s_k$. Set $v' := v\s_k$. Suppose, inductively, that $\s_k$ appeared $t \ge 1$ times in $v$, and that we have identified at least $t-1$ distinct position pairs in $v$. Because $t \ge 1$, there are positions $i$ and $j$ with $i \le k < j$ such that $v(i) > k \ge v(j)$ (see Theorem~\ref{thm:minimal} and \cite[Lemma 2.8]{tenner repetition}). Because we are working with a reduced decomposition, we must have $\ell(v') > \ell(v)$, so $v(k) < v(k+1)$. Recall Remark~\ref{rem:straddling pair marks extremes}, and note that $v'(k) > v(k)$ and $v'(k+1) < v(k+1)$. Thus, for any position pair $(x,y)$ in $v$, even if $x = k$ or $y = k+1$, this $(x,y)$ will also be a position pair in $v'$.

It remains to show that $v'$ will have a position pair (perhaps more than one) that was not a position pair in $v$. There are three cases to consider.
\begin{itemize}
\item Suppose that $v(k+1) < v(i)$. Then $i < k$. Define $a < k$ to minimize $v(a) > v(k+1)$. Such an $a$ exists because $i < k$ and $v(i) > v(k+1)$. Then
$$\{v'(a) = v(a), v'(k) = v(k+1), v'(k+1) = v(k)\}$$
is an occurrence of $321$ in $v'$, and $(a,k+1)$ is a position pair in $v'$. This $(a,k+1)$ was not a position pair in $v$ because it could not have marked a $321$-pattern or a straddling $3412$-pattern due to the choice of $a$.
\item Similarly, if $v(k) > v(j)$, then $j > k+1$ and we define $b > k+1$ to maximize $v(b) < v(k)$. Then
$$\{v'(k) = v(k+1), v'(k+1) = v(k), v'(b) = v(b)\}$$
is an occurrence of $321$ in $v'$, and $(k,b)$ is a position pair in $v'$. This $(k,b)$ was not a position pair in $v$, due to the choice of $b$.
\item Finally, suppose $v(k+1) > v(i)$ and $v(k) < v(j)$. Then $i < k$ and $j > k+1$, and
$$\{v'(i) = v(i), v'(k) = v(k+1), v'(k+1) = v(k), v'(j) = v(j)\}$$
is an occurrence of $3412$ in $v'$, marked by position pair $(k,k+1)$ in $v'$. Certainly $(k,k+1)$ was not a position pair in $v$, because $v(k) < v(k+1)$.
\end{itemize}

Therefore, all $\s_k$ factors except for the first one will introduce at least one new position pair. Were we to read the reduced decomposition with $\max_k(w)$ copies of $\s_k$, we would find that there must be at least $\max_k(w) - 1$ position pairs in $w$, and so $\max_k(w) - 1 \le \pos_k(w)$, as desired.

A symmetric argument shows the $\max_k(w) - 1 \le \val_k(w)$, and the result follows.
\end{proof}

For some permutations and values of $k$, the bound given in Theorem~\ref{thm:maximal} is sharp. For others, it is not, and there are more straddling pairs than there are repeated appearances of the simple reflection. The ``excess'' straddling pairs are, perhaps, related to the main inequality in \cite{tenner repetition}, although much remains to be understood about this surplus.

\begin{example}\label{ex:max in 4312 and 4321}\
\begin{enumerate}
\item[(a)] Consider $w = 4312$ and $k = 1$. Then $\pos_1(w) = 2$ and $\val_1(w) =1$, so $\max_1(w) \le \min\{2,1\}+1 = 2$. Indeed, $\max_1(w) = 2$, as we saw in Example~\ref{ex:4312}.
\item[(b)] Consider $w = 4321$ and $k = 2$. Then $\pos_2(w) = \val_2(w) = 3$, and so $\max_2(w) \le 3+1 = 4$. In fact, $\max_2(w) = 3 < 4$.
\item[(c)] Returning to the permutation $w = 5273416$ and Table~\ref{ex:pair counts}, and we find excess only for $k=4$; that is, for this $w$, the weak inequality in~\eqref{eqn:max} is an equality if and only if $k = 4$.
\end{enumerate}
\end{example}

One might read Example~\ref{ex:max in 4312 and 4321}(b) and hope that the ten patterns given in \cite[Theorem 3.2]{tenner repetition} and listed at the beginning of this paper are the key to understanding this excess. However, that is not the case, as we see in the next example.

\begin{example}
For $w = 34512$ and for all $k \in [1,4]$, the weak inequality stated in~\eqref{eqn:max} is an equality.
\end{example}

The permutation $n(n-1)\cdots 321 \in \symm_n$ is particularly important, and the following corollary gives a range for the repetition of $\s_k$ in its reduced decompositions.

\begin{corollary}
Fix $n \ge 3$. Let $w_0 \in \symm_n$ be the longest element and fix $k \in [1,n-1]$. Then
$$\min\nolimits_k(w_0) = \min\{k,n-k\} \text{ \ and \ } \max\nolimits_k(w_0) \le k(n-k).$$
\end{corollary}

\begin{proof}
The amount of expatriation at $k$ in $w_0$ is the smaller of $k$ and $n-k$. Thus, by Theorem~\ref{thm:minimal}, $\min_k(w_0) = \min\{k,n-k\}$. On the other hand, the number of position pairs (equivalently, of value pairs) at $k$ is $(k-1)(n-k) + n-k-1$. Thus, by Theorem~\ref{thm:maximal}, $\max_k(w_0) \le (k-1)(n-k) + n-k-1 + 1 = k(n-k)$.
\end{proof}

\section{Fixed repetition}\label{sec:fixed}

As we have seen in the previous sections, $\min_k(w)$ and $\max_k(w)$ behave quite differently. That being said, we can characterize when they coincide by recognizing that this coincidence means that every element of $R(w)$ must have the same number of $\s_k$ factors. In a way, this generalizes the results of Section~\ref{sec:straddling}.

\begin{theorem}\label{thm:fixed}
Fix $w \in \symm_n$ and $k \in [1,n-1]$. Then $\max_k(w) > \min_k(w)$ if and only if there exists a $321$-pattern straddling position $k$ in $w$, with position pair $(i,j)$, such that
$$L_w := \{w(q) > w(i) : i < q \le k\} \hspace{.25in} \text{and} \hspace{.25in} R_w := \{w(q) < w(j) : k < q < j\}$$
satisfy
\begin{itemize}
\item $|L_w| = |R_w|$ and
\item the elements of $L_w$ are in increasing order from left to right in $w$, as are those of $R_w$.
\end{itemize}
\end{theorem}

\begin{proof}
Every element of $R(w)$ has a fixed number of $\s_k$ factors if and only if there is no factor $\s_k\s_{k\pm1}\s_k$ in any element of $R(w)$.

Throughout this proof, if $S$ is a set of real numbers and $r \in \mathbb{R}$, we write ``$S < r$'' (respectively, ``$S > r$'') to mean that all elements of $S$ are less than (resp., greater than) $r$.

First suppose that $w$ has a $321$-pattern as described in the statement of the theorem. Set $x := w(i)$ and $z := w(j)$. Without loss of generality, suppose that the middle value of this pattern occurs to the left of position $k+1$. Choose $h \le k$ to be maximal such that $z < w(h) < x$, and set $y := w(h)$. Thus we can write
$$w = \ub{\cdots \phantom{j}\! x \ A \ y \ B}_k \ \ub{C \ z \phantom{j}\! \cdots}_{n-k}.$$
For the remainder of the proof, we will use the underbrace to distinguish the first $k$ positions from the last $n-k$ positions, but will omit the labels ``$k$'' and ``$n-k$.'' By definition of $y$, we can partition the sets $A$, $B$, and $C$ as follows:
\begin{itemize}
\item $A = A_1 \cup A_2$, where $A_1 > x$ and $A_2 < x$,
\item $B = B_1 \cup B_2$, where $B_1 > x$ and $B_2 < z$, and
\item $C = C_1 \cup C_2$, where $C_1 > z$ and $C_2 < z$.
\end{itemize}
Thus $L_w = A_1 \cup B_1$ and $R_w = C_2$, and hence
\begin{equation}\label{eqn:set sizes in balance}
|A_1| + |B_1| = |C_2|.
\end{equation}

We once again employ the techniques used above, shortening the permutation $w$ in order to write $w = vu$, with $\ell(w) = \ell(v) + \ell(u)$, for which $v$ has some particularly useful form. We will write $w \rightsquigarrow v$ to indicate such a maneuver. Using the definitions of the sets $A_i$, $B_i$, and $C_i$ above, we take the following steps, in which ``$S$,'' for example, is understood to mean ``the elements of the set $S$, in the order in which they appear in $w$.''
\begin{align*}
w = \ \ub{\cdots \phantom{j}\! x \ A \ y \ B} \ \ub{C \ z \phantom{j}\! \cdots}
&\ \rightsquigarrow \ 
\ub{\cdots \phantom{j}\! x \ A_2\ A_1 \ y \ B_2\ B_1} \ \ub{C_2\ C_1 \ z \phantom{j}\! \cdots}\\
&\ \rightsquigarrow \ 
\ub{\cdots \phantom{j}\! A_2\ B_2 \ x \ A_1 \ y \ B_1} \ \ub{C_2 \ z \ C_1 \phantom{j}\! \cdots}\\
&\ \rightsquigarrow \ 
\ub{\cdots \phantom{j}\! A_2\ B_2 \ x\ y \ A_1\ B_1} \ \ub{C_2 \ z \ C_1 \phantom{j}\! \cdots}\\
&\ \rightsquigarrow \ 
\ub{\cdots \phantom{j}\! A_2\ B_2 \ x\ y \ C_2} \ \ub{A_1\ B_1 \ z \ C_1 \phantom{j}\! \cdots}\\
&\ \rightsquigarrow \ 
\ub{\cdots \phantom{j}\! A_2\ B_2 \ C_2 \ x\ y} \ \ub{z \ A_1\ B_1 \ C_1 \phantom{j}\! \cdots} \ =: v
\end{align*}
The position of the underbrace after the fourth transition in this list follows from Equation~\eqref{eqn:set sizes in balance}, meaning that $v$ has a $321$-pattern in positions $\{k-1,k,k+1\}$. Hence $v$ has reduced decompositions with the factor $\s_k \s_{k-1}\s_k$, and others that differ only by changing that factor to $\s_{k-1}\s_k\s_{k-1}$. Therefore $\max_k(v) > \min_k(v)$, and so $\max_k(w) > \min_k(w)$.

Now suppose that $\max_k(w) > \min_k(w)$. Without loss of generality, there is a reduced decomposition of $w$ with a factor $\s_k\s_{k-1}\s_k$, and hence $w \rightsquigarrow r$ with $r(k-1) > r(k) > r(k+1)$. If $w = r$, then this pattern, with position pair $(k-1,k+1)$, certainly satisfies the requirements of the theorem because $L_w = R_w = \emptyset$. Now assume, inductively, that $v$ is a permutation with $\max_k(v) > \min_k(v)$, and in which we can find a $321$-pattern as described in the statement of the theorem. Let the position pair for this pattern be $(i,j)$ and set $x := w(i)$ and $z:= w(j)$. Without loss of generality, suppose that the middle value of this pattern occurs to the left of position $k+1$. Choose $h \le k$ to be maximal such that $z < w(h) < x$, and set $y := w(h)$. Thus we can write
$$v = \ub{\cdots \phantom{j}\! x \ A \ y \ B} \ \ub{C \ z \phantom{j}\! \cdots}.$$
Suppose that $\ell(v\s_h) > \ell(v)$. Then certainly $\max_k(v\s_h) > \min_k(v\s_h)$. Our goal is to show that $v\s_h$ has such a $321$-pattern as well.

Consider, first, $h = k$. There are four cases.
\begin{quote}
\textit{Case 1:} $B = C = \emptyset$.\\
This cannot occur because $y > z$ and $\ell(v\s_k) > \ell(v)$.\\

\noindent \textit{Case 2:} $C = \emptyset$ (and hence $B \neq \emptyset$).\\
Then $R_v = \emptyset$, which means that $L_v = \emptyset$ and $A \cup B < x$. Then the values $x > y > v(k)$ give the desired $321$-pattern in $v\s_k$, with $L_{v\s_k} = R_{v\s_k} = \emptyset$.\\

\noindent \textit{Case 3:} $B = \emptyset$ (and hence $C \neq \emptyset$).\\
Let $t > k+1$ be minimal so that $v(t) < y$. (Such a $t$ exists because $z< y$ appears to the right of position $k+1$.) Then $v(k+1) > y > v(t)$ gives the desired $321$-pattern in $v\s_k$, with $L_{v\s_k} = R_{v\s_k} = \emptyset$.\\

\noindent \textit{Case 4:} $B \neq \emptyset$ and $C \neq \emptyset$.\\
Without loss of generality, we need only worry if $v(k) < x < v(k+1)$. Suppose that this is the case. If $v(k) > z$, then, like before, let $t > k+1$ be minimal so that $v(t) < v(k)$, in which case $v(k+1) > v(k) > v(t)$ gives the desired $321$-pattern in $v\s_k$ (with $L_{v\s_k} = R_{v\s_k} = \emptyset$). On the other hand, suppose that $v(k) < z$. Then, either $v(k+1) > L_v$ and $v(k) < R_v$, and so $x > y > z$ gives the desired $321$-pattern in $v\s_k$ (with $L_{v\s_k} = L_v \cup \{v(k+1)\}$ and $R_{v\s_k} = R_v \cup \{v(k)\}$), or, without loss of generality, $v(k)$ is greater than some element of $R_v$. Let $z' \in R(v)$ be the leftmost such element. Then $v(k+1) > v(k) > z'$ gives the desired $321$-pattern in $v\s_k$, with $L_{v\s_k} = R_{v\s_k} = \emptyset$.
\end{quote}

It remains to consider the permutation $v\s_h$, where $h \neq k$. Lengthening $v$ cannot change the fact that $x > y > z$ forms a $321$-pattern straddling position $k$, so we must worry about the sets $L_v = \{l_1 < \cdots < l_m\}$ and $R_v = \{r_m < \cdots < r_1\}$. Without loss of generality, suppose that multiplying by $\s_h$ affects $L_v$. 
One option for this impact is that $\s_h$ swaps the positions of $l_t$ and $l_{t+1}$. Then $l_{t+1} > l_t > x > z > r_{t+1}$, and $l_{t+1} > l_t > r_{t+1}$ would give the desired $321$-pattern in $v\s_h$, with $L_{v\s_h} = \{l_{t+2} < \cdots < l_m\}$ and $R_{v\s_h} = \{r_m < \cdots < r_{t+2}\}$. 
The other way to affect $L_v$ is to move $x$, meaning that $h \in \{i-1,i\}$. Because $\s_h$ lengthens the permutation, we need only worry about $h = i$, where $x < v(i+1)$. This means that $v(i+1) = l_1 \in L_v$. Then $l_1 > y > r_1$ gives the desired $321$-pattern in $v\s_h$, with $L_{v\s_h} = \{l_2 < \cdots < l_m\}$ and $R_{v\s_h} = \{r_m < \cdots < r_2\}$.
\end{proof}

\section{Future research}\label{sec:open}

There are many directions for future research on this topic, and we highlight two fo them here.

A first direction, as alluded to in Section~\ref{sec:maximal} is to better understand the ``excess'' discussed after the proof of Theorem~\ref{thm:maximal}. Given existing results, like the main theorem of \cite{tenner repetition}, it seems possible that one might be able to measure this surplus, or at least to characterize exactly when the weak inequality in~\eqref{eqn:max} is an inequality.

A second direction for future work is related to the Bruhat order. Repetition, both as analyzed above and as in \cite{tenner repetition}, is a feature of reduced decompositions, and reduced decompositions can be used to define the Bruhat order. It is natural to wonder, then, whether some essence of this connection between repeated letters and particular pattern occurrences is respected by the Bruhat order. For this, we look back to the enumeration of $321$- and $3412$-patterns, studied in \cite{tenner repetition}. Recall, as mentioned at the opening of this work, that the main result of that earlier paper had shown a relationship between this tally and $\supp(w)$. For a covering relation in the Bruhat order, $v \lessdot w$, either $\supp(v) = \supp(w)$, in which case the letter deleted from a reduced decomposition of $w$ to form a reduced decomposition of $v$ was not the only copy of that letter in the product, or $\supp(v) \subsetneq \supp(w)$, in which case it was the only copy of that letter. Theorems~\ref{thm:minimal} and~\ref{thm:maximal}, lead one to suspect that this might affect the number of $321$- or $3412$-patterns in the permutations. Unfortunately, while many examples show that the number of occurrences of these patterns is monotonic with respect to the Bruhat order, this is not always the case. Two counterexamples are $561234\lessdot 651234$ and $32541 \lessdot 52341$. That said, there is substantial evidence connecting $321$- and $3412$-patterns to repeated letters, and the subword property of the Bruhat order strongly suggests that there is more yet to say here.

\section*{Acknowledgements}

I am grateful to Sara Billey, Petter Br\"and\'en, Sylvie Corteel, and Svante Linusson for organizing the 2020 program in Algebraic and Enumerative Combinatorics at the Institut Mittag-Leffler, at which much of the research for this paper occurred. I am also grateful for the thoughtful suggestions and observations of an anonymous referee.

\end{document}